\documentclass[a4paper,11pt,reqno]{amsart}

\usepackage[top=3cm, bottom=3cm, left=3cm, right=3cm]{geometry}

\usepackage[utf8]{inputenx}

\usepackage[english]{babel}
\usepackage{microtype}

\usepackage{mathtools}
\usepackage{amssymb}
\usepackage{amsfonts}
\usepackage{amsthm}
\usepackage{mathrsfs}
\usepackage{stmaryrd}

\usepackage{graphicx}
\usepackage{tikz}
\usepackage{tikz-cd}
\usepackage{dynkin-diagrams}

\usepackage{tensor}
\usepackage{centernot}
\usepackage{faktor}
\usepackage{xfrac}
\usepackage{esint}
\usepackage{bbm}

\usepackage{hyperref}

\usepackage{xcolor}
\usepackage{array, tabularx}
\usepackage{mathdots}
\usepackage{comment}
\usepackage{paralist}
\usepackage{natbib}

\theoremstyle{plain}
\newtheorem{thm}{Theorem}[section]

\newtheorem{lem}[thm]{Lemma}

\theoremstyle{definition}
\newtheorem{defn}[thm]{Definition}
\newtheorem{ex}[thm]{Example}

\theoremstyle{remark}
\newtheorem{rem}[thm]{Remark}

\DeclareMathOperator{\Ric}{Ric}

\newcommand{\bigslant}[2]{{\raisebox{.2em}{$#1$}\left/\raisebox{-.2em}{$#2$}\right.}}

\begin{document}
	\newpage
	
	\title{Futaki invariant on Hopf manifolds}

  \author{Giacomo Perri}

	\address{Giacomo Perri \newline
		\textsc{\indent Institut for Matematik, Aarhus University\newline 
			\indent 8000, Aarhus C, Denmark}}
	\email{g.perri@math.au.dk}

\thanks{The author is supported by the Sapere Aude project $\lq\lq$Conformal geometry: metrics and cohomology" at Aarhus University}

\begin{abstract}
The Futaki invariant is a fundamental tool in K\"ahler geometry representing, for instance, an obstruction to the existence of K\"ahler--Einstein metrics. Recently, it was generalized to any compact complex manifold by Futaki, Hattori, and Ornea. In this paper, we prove that it vanishes on all Hopf manifolds, providing a partial answer to a question raised by Ornea and Verbitsky.
\end{abstract}
\maketitle

\section{Introduction}

The Futaki invariant has been intensively studied in Kähler geometry, representing an obstruction to the existence of Kähler--Einstein metrics (see \citep{Futaki1983}). Defined as a character on the Lie algebra of holomorphic vector fields, the Futaki invariant has been extended naturally to the general Hermitian setting in \citep{FHO}, where a non-K\"ahler counterpart is considered. In that framework, the authors define an integral invariant that vanishes whenever a ``canonical" volume form exists. More specifically, it is an obstruction to the existence of a volume form proportional to the determinant of its Ricci form. However, relatively few non-K\"ahler examples have been worked out in detail, and those that are understood appear naturally in locally conformally Kähler (LCK) geometry, a type of Hermitian non-Kähler geometry, interesting from the conformal point of view. Namely, in \citep{FHO}, it is proved that the invariant vanishes on a particular class of LCK manifolds, called Vaisman manifolds. An explicit example is given in the same paper, showing that the Futaki invariant does not vanish on a blown-up Hopf surface. This suggests that the Futaki invariant is of particular interest in the context of LCK geometry and may encode some global features of LCK manifolds. However, not every LCK manifold admits holomorphic vector fields; for instance, this is the case for Inoue surfaces and, more generally, for Oeljeklaus--Toma manifolds \citep{OT05}. Thus, when holomorphic vector fields do exist, it becomes particularly meaningful to ask for vanishing results within natural subclasses of LCK manifolds.

In \citep{OV24}, the authors suggest that vanishing results may hold within the class of LCK manifolds with potential. An LCK manifold with potential is an LCK manifold with Kähler cover admitting a global, positive, and automorphic Kähler potential.

Motivated by this, the aim of this paper is to prove that the Futaki invariant vanishes on any Hopf manifold in arbitrary dimension, regardless of carrying a Vaisman metric or not (Theorem \ref{vanishingonHopf}). Hopf manifolds are the earliest examples of LCK manifolds with potential \citep{OV23} and provide key structural models in the field. They are defined as the quotient of \(\mathbb{C}^n \setminus \{0\}\) by a biholomorphic contraction centered at zero. In \citep{OV10}, it is shown that each LCK manifold with potential and complex dimension greater than two admits a holomorphic embedding into a linear Hopf manifold. 

In the next section, we provide the necessary background and tools for the main section, which is devoted to proving the vanishing of the Futaki invariant on Hopf manifolds. This is achieved by reducing the computation on a generic Hopf manifold to its corresponding diagonal Hopf manifold, which is Vaisman. This diagonal Hopf manifold is obtained by taking the linear diagonal part of the contraction defining the generic Hopf manifold, after a suitable biholomorphic change of coordinates.

\section*{Acknowledgements}
The author is grateful to Alexandra Otiman for first pointing out this problem, for carefully reading the manuscript, and for providing useful suggestions. The author also thanks Liviu Ornea and Cristiano Spotti for helpful remarks.

\section{Preliminaries}

In this section, we briefly recall the essential ingredients needed to define the Futaki invariant in the context of Hermitian geometry and its geometric interpretation; for a complete treatment, we refer the reader to \citep{FHO}. We then provide a concise overview of the main definitions and results concerning LCK geometry. For a complete discussion, see \citep{OV24}.

\subsection{The Futaki invariant in complex geometry}
Let $M$ be a compact complex manifold with a volume form $\Omega$. On a complex manifold, giving a volume form $\Omega$ is equivalent to giving a Hermitian metric $H_\Omega$ on the canonical bundle $K_M$. If $\nu =dz$ is a local section of $K_M$, the Hermitian metric $H_\Omega$ on $K_M$ is given by $H_\Omega(\nu,\nu) := \frac{\nu \wedge \overline{\nu}}{\Omega}$.
The Ricci form $\Ric_\Omega$ of the volume form $\Omega$ is defined as 
$$\Ric_\Omega:=-i\,\Theta\big(K_M,H_\Omega),$$
where $\Theta\big(K_M,H_\Omega)$ denotes the Chern curvature of $K_M$ respect to $H_\Omega$.
In particular, the Ricci form associated with the volume form of a Hermitian metric $\omega$ is the Chern--Ricci curvature of $\omega$, $\textsl{i.e.}$,
$$\Ric_{\omega^n}=\text{CRic}_\omega.$$
If locally $$\Omega \simeq_{\textrm{loc}} a \, idz_1\wedge d\overline{z}_1\wedge\cdots\wedge idz_n\wedge d\overline{z}_n,$$
then the associated Ricci form is given by 
$$\Ric_\Omega \simeq_{\textrm{loc}} -i\,\partial \overline{\partial} \log a.$$
Suppose we have a holomorphic vector field $V \in H^0\big(M, T^{1,0}_M\big)$ on $M$. The divergence of $V$ with respect to the volume form $\Omega$ is defined as 
$$\text{div}_\Omega V = \frac{\mathcal{L}_V \Omega}{\Omega} = \frac{d (\iota_V \Omega)}{\Omega} = \frac{\partial (\iota_V \Omega)}{\Omega},$$
where in the last equality we used that $V$ is holomorphic. By $\mathcal{L}_V \Omega$ we mean the Lie derivative of $\Omega$ along $V$, and by $\iota_V$ we mean the contraction with respect to $V$. Locally, the divergence is given by 
$$\text{div}_\Omega V \simeq_{\textrm{loc}} V(\log a) + \sum_{l=1}^n \frac{\partial V_l}{\partial z_l},$$
where $V \simeq_{\textrm{loc}} \sum_{l=1}^n V_l \frac{\partial}{\partial z_l}$.

\begin{defn}
   (\citep{FHO}) The $\textbf{Futaki invariant}$ $\text{F}_M$ of $M$ is a map from the holomorphic vector fields on $M$ to $\mathbb{C}$,
   $$\text{F}_M:H^0\big(M, T^{1,0}_M\big) \to \mathbb{C}$$
   defined as 
$$\text{F}_M(V) := \int_M \text{div}_\Omega V \, \Ric_\Omega^n.$$
\end{defn}
As shown in \citep{FHO}, this definition does not depend on the choice of the volume form $\Omega$ on $M$.\\
The following remark explains in what sense the Futaki invariant represents an obstruction to the existence of a volume form proportional to the determinant of its Ricci form.
\begin{rem} The following equality holds
\begin{equation}\label{Interpretation}
    \text{F}_M(V) = \int_M \text{div}_\Omega V \, \Ric_\Omega^n = -\int_M V\left(\frac{\Ric_\Omega^n}{\Omega}\right) \Omega.
\end{equation}
Using the properties of the Lie derivative $\mathcal{L}$, we find
\begin{align*}
    V\left(\frac{\Ric_\Omega^n}{\Omega}\right) \Omega &= \iota_V d\left(\frac{\Ric_\Omega^n}{\Omega}\right) \Omega\\ 
    &= \mathcal{L}_V \left(\frac{\Ric_\Omega^n}{\Omega}\right) \Omega-d \,\iota_V \left(\frac{\Ric_\Omega^n}{\Omega}\right) \Omega  \\
    &= \mathcal{L}_V \big(\Ric_\Omega^n\big) - \mathcal{L}_V \Omega\,\frac{\Ric_\Omega^n}{\Omega}  \\
    &= \mathcal{L}_V \big(\Ric_\Omega^n\big) - \text{div}_\Omega V\Ric_\Omega^n  = d \,\iota_V \Ric_\Omega^n - \text{div}_\Omega V\Ric_\Omega^n .
\end{align*}
From the compactness of $M$ and Stokes' theorem, $(\ref{Interpretation})$ follows. 
\end{rem}

\subsection{LCK manifolds}
\begin{defn}
    A $\textbf{Locally Conformally Kähler (LCK)}$ manifold $(M,J,\omega,\theta)$ is a complex Hermitian manifold, $\dim_\mathbb{C}(M) > 1$, with a Hermitian metric $\omega$ satisfying $d\omega = \theta \wedge \omega$, where $\theta$ is a closed 1-form, called the $\textbf{Lee form}$ of $\omega$.
\end{defn}
There are several equivalent ways to define an LCK manifold. A Hermitian manifold $(M, J, \omega)$ is LCK if there exists an open cover $\{U_i\}_{i \in I}$ of $M$ and smooth functions $f_i: U_i \to \mathbb{R}$ such that $e^{-f_i} \omega_{|_{U_i}}$ is Kähler on $U_i$ for each $i\in I$. Another equivalent definition is as follows: a complex manifold $(M, J)$ is LCK if there exists a Kähler metric $\tilde{\omega}$ on the universal cover $\tilde{M}$ of $M$ such that for each element $\gamma \in \Gamma$ in the deck transformation group, there exists a positive constant $c_\gamma > 0$ such that $\gamma^* \tilde{\omega} = c_\gamma \tilde{\omega}$, \textsl{i.e.}, $\Gamma$ acts on $(\tilde{M}, \tilde{\omega})$ by Kähler homotheties. 
In this situation, a smooth function $f$ on $\tilde{M}$ is called automorphic if for each $\gamma \in \Gamma$, $\gamma^* f = c_\gamma f$.

If the Lee form $\theta$ is exact, with $\theta = df$, then $e^{-f}\omega$ is Kähler and $M$ is called globally conformally Kähler (GCK). We assume $[\theta] \neq 0$ to avoid this situation.
\begin{defn}
    An LCK manifold $(M, J, g)$ is called a \textbf{Vaisman manifold} if $\nabla \theta = 0$, where $\nabla$ is the Levi--Civita connection of $g$.
\end{defn}
On a Vaisman manifold, the Lee vector field $\theta^\#$ is a holomorphic vector field, and therefore, for Vaisman manifolds, it is meaningful to compute the Futaki invariant.
\begin{ex}\label{Hopf_example}
A Hopf manifold $H = H_{\gamma}$ is defined as the quotient 
\begin{equation}\label{HOPFMANIFOLDS}
    H := \bigslant{\mathbb{C}^n \setminus \{0\}}{\langle\gamma\rangle},
\end{equation}
where 
$$\gamma: \mathbb{C}^n \to \mathbb{C}^n$$
is a biholomorphic contraction centered at zero. Hopf manifolds are LCK manifolds \citep{OV23}. If $\gamma = A \in GL(n, \mathbb{C})$ is diagonal, then $H_A$ is Vaisman \citep{OV16}.
\end{ex}
We conclude this section by recalling the vanishing result proved in \citep{FHO} for the Futaki invariant. 
\begin{thm}[{\citep[Theorem 1.2]{FHO}}]\label{Vaisamnvanishingg} 
    The Futaki invariant vanishes on compact Vaisman manifolds.
\end{thm}    
\begin{rem}
    In \citep{FHO}, through a direct computation, it is shown that on a blow-up of a diagonal Hopf surface at a point, the Futaki invariant is not zero, and thus, there are examples of LCK non-Vaisman manifolds (see \citep{OVV13}) with non-vanishing Futaki invariant.
\end{rem}

\section{Vanishing of the Futaki invariant on Hopf manifolds}
This section is devoted to the proof of our main result.
\begin{thm}\label{vanishingonHopf}
    The Futaki invariant vanishes on Hopf manifolds.
\end{thm}
It is organized as follows. 
We begin by describing a general Hopf manifold in a normal form. Then, we study its holomorphic vector fields and conclude with the proof of Theorem \ref{vanishingonHopf}.\\

Let $H = H_{\gamma}$ be a Hopf manifold. We describe $H = H_{\gamma}$ in a more manageable form for our purpose. Following \citep[Appendix A]{H} and using the Poincaré--Dulac theorem \citep[p.181]{Arn88}, we can assume that 
$\gamma$ is a $\mu$-resonant biholomorphic polynomial map whose derivative $\gamma'(0)$ at zero has eigenvalues 
$$\mu = (\mu_1, \mu_2, \ldots, \mu_n)$$
with 
$$0 < |\mu_1| \leq |\mu_2| \leq \ldots \leq |\mu_n| < 1.$$
Recall that a $\mu$-resonant (multiplicative) monomial is a polynomial map from $\mathbb{C}^n$ to 
$\mathbb{C}^n$ of the form 
$$z \mapsto az^m e_s,$$
such that 
$$\mu_s = \mu^m = \mu_1^{m_1} \mu_2^{m_2} \ldots \mu_n^{m_n},$$
where 
$m = (m_1, m_2, \ldots, m_n) \in \mathbb{N}^n$, $a \in \mathbb{C}$, and $\{e_s\}_s$ is the canonical 
basis of $\mathbb{C}^n$. A $\mu$-resonant polynomial map $f: \mathbb{C}^n \to \mathbb{C}^n$ is a sum of 
$\mu$-resonant monomials.\\
We can assume that $\gamma$ has a normal form 
$$\gamma(z) = \big(\gamma_1(z), \gamma_2(z), \ldots, \gamma_n(z)\big),$$
 where 
$$\gamma_j(z) = \mu_j z_j + P_j(z_{j+1}, \ldots, z_n)$$ 
and 
$$P_j(z_{j+1}, \ldots, z_n)=\sum_{m}a_{mj}z^m$$
is a polynomial in the variables $z_{j+1}, \ldots, z_n$.\\
The following lemma gives us information about the holomorphic vector fields of $H$ and it will be crucial in the proof of Theorem \ref{vanishingonHopf}. It is proved in \citep{H}. Instead of giving a full proof, we outline the main ideas.
\begin{lem}[{\citep[A.4]{H}}]\label{Vectorfields}
    Let
    $$d_\gamma(z) := (\mu_1 z_1, \mu_2 z_2, \ldots, \mu_n z_n)$$
    be the diagonal part of $\gamma$. Then, there is a natural injective map
\begin{equation}\label{inclusion} 
    H^0\big(H_\gamma, T^{1,0}_{H_\gamma}\big) \hookrightarrow H^0\big(H_{d_\gamma}, T^{1,0}_{H_{d_\gamma}}\big),
\end{equation}
\textsl{i.e.}, each holomorphic vector field on $H_\gamma$ induces a holomorphic vector field on $H_{d_\gamma}$.
\end{lem}

Recall that  
$$H^0\big(H_\gamma, T^{1,0}_{H_\gamma}\big) = \big\{V \in H^0\big(\mathbb{C}^n, T^{1,0}_{\mathbb{C}^n}\big) \mid \gamma_* V = V\big\}.$$  
We denote by $ g_\mu $ the Lie subalgebra of the Lie algebra of holomorphic vector fields on $ \mathbb{C}^n $, consisting of vector fields 
$$V = \sum a^s_m z^m \frac{\partial}{\partial z_s},$$
such that for each monomial, we have $ \mu_s = \mu^m $ (these vector fields are called multiplicatively $\mu$-resonant).
Observe that, $g_\mu$ coincides with the holomorphic vector fields that are $d_{\gamma}$-invariant, 
\textsl{i.e.},  
$$g_\mu = H^0\big(H_{d_\gamma}, T^{1,0}_{H_{d_\gamma}}\big).$$  
Let $g_\mu^\perp$ be the set of $V \in H^0\big(\mathbb{C}^n, T^{1,0}_{\mathbb{C}^n}\big)$ of the form  
$$\sum a^s_m z^m \frac{\partial}{\partial z_s},$$
where each monomial $a^s_m z^m$ is not $\mu$-resonant, \textsl{i.e.}, $\mu_s\ne\mu^m$.  
In \citep[A.4]{H}, it is shown that  
\begin{equation}\label{gmu}
    \text{Id} - \gamma_*(g_\mu) \subset g_\mu
\end{equation}
and that  
\begin{equation}\label{gmuiso}
    \text{Id} - \gamma_* : g_\mu^\perp \to g_\mu^\perp
\end{equation}
is an isomorphism of $g_\mu^\perp$. The $\mu$-resonant condition and the fact that each $\mu_i$ has modulus strictly less than one imply that, up to complex multiples, there is only a finite number of $\mu$-resonant monomials; hence, $g_\mu$ has finite dimension as a complex vector space. In particular, we have
$$H^0\big(\mathbb{C}^n, T^{1,0}_{\mathbb{C}^n}\big) = g_\mu \oplus g_\mu^\perp.$$
We deduce that if $\gamma_* V = V$, then $V \in g_\mu$.
\begin{rem}
    We did not use the isomorphism given in $(\ref{gmuiso})$ but only property $(\ref{gmu})$, the fact that $\text{Id} - \gamma_*(g_\mu^\perp) \subset g_\mu^\perp$, and that $\text{Id} - \gamma_* \big|_{g_\mu^\perp}$ is injective, all properties that can be verified directly.
\end{rem}

Before proving Theorem \ref{vanishingonHopf}, we need the following technical lemma.
\begin{lem}
Up to biholomorphism, there exists a fundamental domain 
$D \subset \mathbb{C}^n \setminus \{0\}$ for the action of $\gamma$ 
such that there exists an open set $U \subset D$ with the following property:
\begin{equation}\label{property}
    \Gamma(c):=\big\{z \in \mathbb{C}^n : |z| \in [c,1]\big\}
    \subsetneq U \subsetneq \overline{U}
    \subsetneq D \setminus \partial D,
\end{equation}
for some real number $c<1$ (ideally we take $1-c\ll 1$).
\end{lem}
\begin{proof}
We prove that, up to biholomorphism, we may assume that if $|z|=1$, then $|\gamma(z)|<1$. 
If $\gamma$ is linear and diagonal, there is nothing to prove. We may therefore assume that we are not in this case. If $t \in \mathbb{C}^*$ 
is fixed, the biholomorphism of $\mathbb{C}^n$ given by 
$$d_t(z) = \big(t z_1, t^2 z_2, \ldots, t^n z_n\big)$$
induces a biholomorphism 
\begin{equation}\label{d_t}
    d_t: \bigslant{\mathbb{C}^n \setminus \{0\}}{\langle\gamma\rangle}\to\bigslant{\mathbb{C}^n \setminus \{0\}}{\langle d_t \gamma d_t^{-1}\rangle}
\end{equation}
since $d_t\circ\gamma = (d_t \gamma d_t^{-1}) d_t$. We have that 
\begin{align*}
    \big(d_t \gamma d_t^{-1}(z)\big)_j &= t^j \gamma_j\left(\frac{z_1}{t}, \frac{z_2}{t^2}, \ldots, \frac{z_n}{t^n}\right) \\
    &= \mu_j z_j + \sum_{m} a_{mj} \frac{t^j}{t^{\sum_{l=j+1}^n l m_l}} z^m,
\end{align*}
where for each $m$ in the sum, there exists an $m_{\tilde{l}} \neq 0$ with $\tilde{l} \geq j+1$, 
and thus $\sum_{l=j+1}^n l m_l - j > 0$.
Suppose that $z \in \mathbb{S}^{2n+1}$ and $t \gg 1$ is a real number. Then  
\begin{align*}
    \big|d_t \gamma d_t^{-1}(z)\big|^2 &\leq \sum_{j=1}^n |\mu_j|^2 |z_j|^2 + \frac{1}{t} A(z) \\
    &\leq |\mu_n|^2+\frac{C}{t}< 1,
\end{align*}
where $A$ is a smooth function with real values and $C:=\max_{\mathbb{S}^{2n+1}}|A|$.\\
By compactness and analogous computations, possibly by taking $t$ larger, 
we can find a fundamental domain 
$D \subset \mathbb{C}^n \setminus \{0\}$ 
for the action of $\gamma_t := d_t \gamma d_t^{-1}$ 
such that there exists an open set $U \subset D$ satisfying \eqref{property}, namely
\[
    \Gamma(c):=\big\{z \in \mathbb{C}^n : |z| \in [c,1]\big\}
    \subsetneq U \subsetneq \overline{U}
    \subsetneq D \setminus \partial D,
\]
for some real number $c<1$.
\end{proof}
In the following, we assume that $\gamma$ satisfies property \eqref{property}. 
Then, the boundary of $D$ consists of two connected components, denoted by 
$D_1$ and $D_2$. 
We denote by $D_1$ the component contained in the ball $\mathbb{B}(0,c)$, 
and by $D_2$ the other one.

\begin{proof}[Proof of Theorem \ref{vanishingonHopf}]
To compute the Futaki invariant, we need to determine all holomorphic vector fields and construct a specific volume form on $H$. Lemma \ref{Vectorfields} provides the necessary ingredients concerning the holomorphic vector fields. We begin the proof by constructing an $\textsl{ad hoc}$ volume form for $H$.

In what follows, we use the notation 
$$dz \wedge d\overline{z} = i^n \, dz_1 \wedge d\overline{z}_1 \wedge \ldots \wedge  \, dz_n \wedge d\overline{z}_n.$$
Observe that a volume form $v = h \, dz \wedge d\overline{z}$ on $\mathbb{C}^n \setminus \{0\}$ is well-defined on $H$ if and only if it is $\gamma$-invariant, \textsl{i.e.}, $\gamma^*v = v$, and then if and only if 
\begin{equation}\label{volumecondition}
    h\big(\gamma(z)\big)|\mu_1\mu_2\ldots\mu_n|^2 = h(z)\,\,\,\forall\,z\in\mathbb{C}^n \setminus \{0\}.
\end{equation}
Thus, finding a volume form on $H$ is equivalent to finding a smooth positive function on $\mathbb{C}^n\setminus\{0\}$ satisfying the equivariance relation $(\ref{volumecondition})$.\\
As a first step, we construct a cut-off function $\psi$ on $D$, constant on each component of the complement of $\Gamma(c)$, such that $$\psi(D_1)|\mu_1\mu_2\ldots\mu_n|^2 = \psi(D_2).$$ 
Since $\mathbb{C}^n\setminus\{0\}=\sqcup_{i\in\mathbb{Z}}D_i$, where $D_0$ is $D$ and $D_i=\gamma^{i}(D_0)$, we will be able to replicate $\psi$ on each $D_i$ by the action of $\mathbb{Z}$ and define a global function $f$ satisfying $(\ref{volumecondition})$. We define $\psi$ as
$$\psi(z) := v -g\left(\frac{\|z\|^2 - c^2}{1-c^2 }\right),
$$
where (as in \citep[p. 143]{Tu})
$$
g(t) := \frac{\phi(t)}{\phi(t) + \phi(1-t)}, \quad \phi(t) := 
\begin{cases} 
e^{-\frac{1}{t}}, & t > 0, \\
0, & t \leq 0,
\end{cases}
$$
and $v>1$ is a real number such that 
\begin{align*}
    \frac{v-1}{v}=|\mu_1\mu_2\ldots\mu_n|^2.
\end{align*}
Once we have this function $\psi$, by its definition, we can define a positive smooth function $f$ on $\mathbb{C}^n \setminus \{0\}$ via $\mathbb{Z}\gamma^*$ that satisfies $(\ref{volumecondition})$ as follows:  
$$f\big(\gamma^k(z)\big) := \left( \frac{1}{|\mu_1\mu_2\ldots\mu_n|^2}\right)^k\psi(z) \quad \forall \,z \in D.$$  
We emphasize that $f = \psi$ on $D$. By construction, $f \, dz \wedge d\overline{z}$ provides a volume form on $H$.\\

To determine the Futaki invariant on $H$, we use the fact that it is zero on diagonal Hopf manifolds. Diagonal Hopf manifolds are Vaisman (Example \ref{Hopf_example}) and the Futaki invariant on Vaisman compact manifolds vanishes (Theorem \ref{Vaisamnvanishingg}). The strategy will be to compare the Futaki invariant of $\mathbb{C}^n\setminus\{0\}/\gamma$ to the one of $\mathbb{C}^n\setminus\{0\}/d_{\gamma}$, which we know vanishes. To this aim, we find a spherical shell included in both fundamental domains and we construct the volume forms $f dz \wedge d\overline{z}$ and $\tilde{f} dz \wedge d\overline{z}$ in such a way that $f$ coincides with $\tilde{f}$ on an annulus $\Gamma(c)$ as described above. Consider therefore the diagonal Hopf manifold $\tilde{H}$ given by
$$\tilde{H} := \bigslant{\mathbb{C}^n \setminus \{0\}}{\langle d_\gamma \rangle}.$$
As for $H$, the condition 
$$h\big(d_\gamma(z)\big) |\mu_1\mu_2\ldots\mu_n|^2 = h(z)$$
defines a volume form on $\tilde{H}$ and we consider a fundamental domain $\tilde{D}$ such that there exists an open set $\tilde{U} \subset \tilde{D}$ with the following property:  
$$
\Gamma(c) = \big\{z \in \mathbb{C}^n : |z| \in [c,1]\big\} \subsetneq \tilde{U}  \subsetneq \overline{\tilde{U}} \subsetneq \tilde{D} \setminus \partial \tilde{D},
$$  
where $1-c$ is taken sufficiently small.  
We then construct a volume form $\tilde{f} \, dz \wedge d\overline{z}$ for $\tilde{H}$ starting from the function $\psi$, as done for $H$.
The crucial point is the following. By construction, 
\begin{equation}\label{ff}
    f = \tilde{f}\,\text{    on    }\, \Gamma(c) \subset D \cap \tilde{D},
\end{equation}
and outside $\Gamma(c)$, $f$ and $\tilde{f}$ are constant in each connected component of $ D \setminus \Gamma(c) $ and $ \tilde{D} \setminus \Gamma(c) $, respectively. Notice that despite $f$ and $\tilde{f}$ coincide on $\Gamma(c)$, they are different as global functions, since $\gamma\big(\Gamma(c)\big)\neq d_{\gamma}\big(\Gamma(c)\big)$.\\
From what has been said above, each holomorphic vector field on $H$ induces a holomorphic vector field on $\tilde{H}$. In conclusion, if we consider the volume forms $\Omega(f) := f \, dz \wedge d\overline{z}$ on $H$ and $\Omega(\tilde{f}) := \tilde{f} \, dz \wedge d\overline{z}$ on $\tilde{H}$ respectively, then by $f$ being constant outside $\Gamma(c)$ and by relation $(\ref{ff})$, the Futaki invariant $\text{F}_H$ on $H$ is given by
\begin{align*}
    \text{F}_H(V) &= \int_D \text{div}_{\Omega(f)}(V) \, \left(\text{Ric}_{\Omega(f)}\right)^n\\
    &=(-i)^n \int_D \text{div}_{\Omega(f)}(V) \, \big(\partial \overline{\partial} (\log f)\big)^n \\
    &= (-i)^n\int_{\Gamma(c)} \text{div}_{\Omega(f)}(V) \, \big(\partial \overline{\partial} (\log f)\big)^n \\
    &=(-i)^n \int_{\Gamma(c)} \text{div}_{\Omega(\tilde{f})}(V) \, \big(\partial \overline{\partial} (\log \tilde{f})\big)^n \\
    &=(-i)^n \int_{\tilde{D}} \text{div}_{\Omega(\tilde{f})}(V) \, \big(\partial \overline{\partial} (\log \tilde{f})\big)^n\\
    &= \text{F}_{\tilde{H}}(V) = 0
\end{align*}
for each holomorphic vector field $V$ on $H$, and then the Futaki invariant on $H$ is zero.
\end{proof}

\bibliographystyle{alpha}

\end{document}